\documentclass{article}
\usepackage{amssymb,amscd,amsthm, verbatim,amsmath,color,fancyhdr, mathrsfs}
\usepackage{tikz}
\usepackage{graphicx}
\usepackage{turnstile}
\usepackage{mathtools}
\usepackage{bm}
\usepackage{inconsolata}  
\usepackage{comment}
\usepackage{todonotes}

\usepackage{mathtools}
\usepackage[colorlinks=true,citecolor=blue]{hyperref}

\usepackage[all]{xy}
\usepackage{float}
\usepackage{bbold}
\usepackage{enumerate}
\usepackage{listings}

\newcommand{\ZZ}{\mathbb{Z}}

\DeclareMathOperator{\impro}{\mu_{int}}

\newtheorem{thm}{Theorem}[section]

\newtheorem{lemma}[thm]{Lemma}
\newtheorem{claim}{Claim}
\newtheorem{prop}[thm]{Proposition}

\theoremstyle{definition}

\newtheorem{remark}[thm]{Remark}

\usepackage[normalem]{ulem} 
\usepackage{bm}
\usepackage[mathlines]{lineno}

\newcommand*\patchAmsMathEnvironmentForLineno[1]{
	\expandafter\let\csname old#1\expandafter\endcsname\csname #1\endcsname
	\expandafter\let\csname oldend#1\expandafter\endcsname\csname end#1\endcsname
	\renewenvironment{#1}
	{\linenomath\csname old#1\endcsname}
	{\csname oldend#1\endcsname\endlinenomath}}
\newcommand*\patchBothAmsMathEnvironmentsForLineno[1]{
	\patchAmsMathEnvironmentForLineno{#1}
	\patchAmsMathEnvironmentForLineno{#1*}}
\AtBeginDocument{
	\patchBothAmsMathEnvironmentsForLineno{equation}
	\patchBothAmsMathEnvironmentsForLineno{align}
	\patchBothAmsMathEnvironmentsForLineno{flalign}
	\patchBothAmsMathEnvironmentsForLineno{alignat}
	\patchBothAmsMathEnvironmentsForLineno{gather}
	\patchBothAmsMathEnvironmentsForLineno{multline}
}

\title{The interval coloring impropriety of planar graphs}

\date{} 
\usepackage[utf8]{inputenc}
\usepackage{authblk} 

\usepackage{array}

\author{Seunghun Lee\footnote{
Department of Mathematical Sciences, KAIST(Korea Advanced Institute of Science and Technology), South Korea. \texttt{seunghun.math@kaist.ac.kr}}}

\begin{document}
	\maketitle

\begin{abstract}
For a graph $G$, we call an edge coloring of $G$ an \textit{improper} \textit{interval edge coloring} if for every $v\in V(G)$ the colors, which are integers, of the edges incident with $v$ form an integral interval.  The \textit{interval coloring impropriety} of $G$, denoted by $\impro(G)$, is the smallest value $k$ such that $G$ has an improper interval edge coloring where at most $k$ edges of $G$ with a common endpoint have the same color.

The purpose of this note is to communicate solutions to two previous questions on interval coloring impropriety, mainly regarding planar graphs. First, we prove $\impro(G) \leq 2$ for every outerplanar graph $G$. This confirms the conjecture by Casselgren and Petrosyan in the affirmative. Secondly, we prove that for each $k\geq 2$, the interval coloring impropriety of $k$-trees is unbounded. This refutes the conjecture by Carr, Cho, Crawford, Ir\v si\v c, Pai and Robinson.

\end{abstract}
\section{Introduction} \label{sec_intro}
Let $G=(V,E)$ be a graph. For a vertex $v \in V$, we denote the set of the edges of $G$ incident with $v$ by $E_v$. We call an edge coloring $c:E \to \ZZ$ of $G$, which is not necessarily proper, an \textit{improper} (or \textit{defective}) \textit{interval edge coloring} if the image $c(E_v)$ forms an integral interval for every $v\in V$. This coloring notion was investigated by Casselgren and Petrosyan \cite{impropriety_casselgren} 
as a defective version of (proper) interval edge colorings introduced by Asratian and Kamalian \cite{proper_interval_original_russian, proper_interval_original_English}. This defective coloring version was first considered by Hud\'ak, Kardo\v s, Madaras and Vrbjarov\'a \cite{improper_interval_hudak} but with a different focus from that of \cite{impropriety_casselgren} (to check the general literature on \textit{improper}, or \textit{defective}, colorings, the reader might want to consult the recent surveys \cite{defective_cowen, defective_wood}).

While the original proper interval edge coloring model was motivated by the real-world problem to determine timetables without gaps for teachers and classes, the defective coloring model considers a similar problem which allows some level of conflict. In \cite{impropriety_casselgren}, such conflict is modeled by the \textit{interval coloring impropriety} (or just \textit{impropriety}) of a graph $G$, denoted by $\impro(G)$, which is the smallest value $k$ such that $G$ has an improper interval edge coloring where at most $k$ edges of $G$ with a common endpoint have the same color, or equivalently, $|c^{-1}(i) \cap E_v|\leq k$ for any color $i \in \ZZ$ and any vertex $v\in V$. 
In the same paper, the authors provided lower and upper bounds on impropriety for various classes of graphs. The same line of study was continued in subsequent papers \cite{carr2024intervalcoloringimproprietygraphs, impropriety_complete_multipartite}.

\medskip 

The purpose of this note is to communicate solutions to two previous questions on impropriety, mainly regarding planar graphs. The first result is on outerplanar graphs. Recall that an \textit{outerplanar graph} is a graph that has a planar drawing where all vertices belong to the outerface of the drawing.

Casselgren and Petrosyan conjectured that $\impro(G) \leq 2$ for any outerplanar graph $G$ \cite[Conjecture 4.10]{impropriety_casselgren} and provided a supporting case when $\Delta(G) \leq 8$. This was later improved in \cite{carr2024intervalcoloringimproprietygraphs}. Our first main theorem confirms the conjecture in the affirmative.

\begin{thm} \label{thm_outer}
	For any outerplanar graph $G=(V,E)$, we have $\impro(G)\leq 2$. That is, there is an improper interval edge coloring $c:E\to \ZZ$ such that for any vertex $v\in V$ and any color $i\in \ZZ$, $|c^{-1}(i) \cap E_v|\leq 2$. 
\end{thm}

Note that Theorem \ref{thm_outer} gives an infinite family of planar graphs with bounded impropriety.

\smallskip

The second result is on $k$-trees. We first define a $k$-tree by recursion: $K_{k+1}$ is $k$-tree, and a $k$-tree on $n>k+1$ vertices is obtained by joining a new vertex to any $k$ pairwise adjacent vertices of a $k$-tree on $n-1$ vertices. In a particular case when $k=2$, note that every maximal outerplanar graph, which forms a triangulation of a polygon, is a 2-tree, and every 2-tree is planar.

In \cite[Conjecture 13]{carr2024intervalcoloringimproprietygraphs}, Carr, Cho, Crawford, Ir\v si\v c, Pai and Robinson conjectured that $\impro(G)\leq 2$ for any 2-tree $G$ after providing some supporting cases. In the same paper, the authors further asked at \cite[Question 27]{carr2024intervalcoloringimproprietygraphs} if the impropriety of a $k$-tree can bounded by a function on $k$. The following theorem refutes their conjecture on 2-trees and gives a negative answer to their general question on $k$-trees.

\begin{thm} \label{thm_2-tree}
Let $k\geq 2$ be a fixed positive integer. For any positive integer $N$, there is a $k$-tree $G$ with $\impro(G) \geq N$. 
\end{thm}

In particular, our construction for Theorem \ref{thm_2-tree} provides a family of planar graphs with arbitrarily large impropriety. This shows a marked contrast within the class of all planar graphs regarding interval coloring impropriety.

\begin{remark}
Several constructions of bipartite graphs with unbounded impropriety (so they are not $k$-trees for $k \geq 2$) were previously
 obtained at \cite[Subsection 3.2]{impropriety_casselgren}. In fact, their construction from the so-called \textit{Hertz's graphs}, first described in \cite{hertz_graph}, gives another construction of planar graphs with unbounded impropriety and shares the same spirit with our construction at Section \ref{sec_k-tree}. 
 \end{remark}

This note is organized as follows. In Section \ref{sec_outerplanar}, we prove Theorem \ref{thm_outer}. In Section \ref{sec_k-tree}, we prove Theorem \ref{thm_2-tree}. In what follows, all graphs are assumed to be simple  unless stated otherwise.

\section{The impropriety of an outerplanar graph} \label{sec_outerplanar}

In this section we prove Theorem \ref{thm_outer}. First, we recall known facts about outerplanar graphs. An \textit{outerplanar graph} is a  graph that has a planar drawing where all vertices belong to the outerface of the drawing. An outerplanar graph $G$ is called \textit{maximal outerplanar} if no edge can be added to $G$ without losing outerplanarity. The following fact is well-known, for example, see \cite{pach_outerplanar} and \cite{pach_geometric_graph_survey}.

\begin{prop} \label{prop_triangulation_polygon}
A maximal outerplanar graph has a straight line drawing which gives a triangulation of a convex polygon.
\end{prop}

\begin{figure}[h]
	\centering
	\includegraphics[totalheight=4cm]{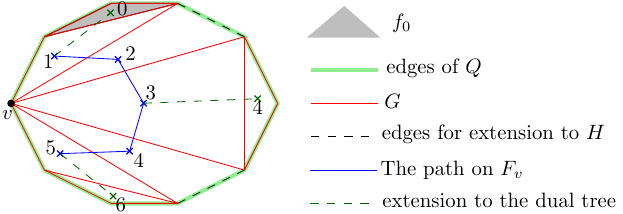}
	\caption{Illustration of the proof of Theorem \ref{thm_outer}. The numbers at each face measures the distance $d_T(f,f_0)$.
	}
	\label{fig:outerplanar}
\end{figure}

\begin{proof}[Proof of Theorem \ref{thm_outer}]
Let $G=(V,E)$ be a given outerplanar graph, and $H$ be a maximal outerplanar graph which contains $G$ as a subgraph. By Proposition \ref{prop_triangulation_polygon}, we have a straight line drawing of $H$ which gives a triangulation of a convex polygon $Q$. We find a straight line drawing of $G$ inside that drawing. Now we consider the following two separate cases.

\medskip 
	
\textbf{(Case 1) When $G$ has all the boundary edges of $Q$:} In this case, the bounded faces of the drawing of $G$ partition $Q$. Let us denote the dual graph of $G$ by $G^*$ with respect to the drawing (note that $G^*$ might have parallel edges incident with the outerface $O$). Also note that the induced subgraph of $G^*$ on the bounded faces of $G$ is a tree. We denote this subgraph by $T$. Fix a bounded face $f_0$. For every bounded face $f$, we measure the distance $d_T(f,f_0)$ in the tree $T$.
	
Let $F_v$ be the set of bounded faces which are incident with $v$. We have the following claim.

\begin{claim}
	For any vertex $v\in V$, there is exactly one face in $F_v$ which attains the minimum $\min_{f\in F_v} d_T(f, f_0)$. 
\end{claim}

There is nothing to prove when $|F_v|=1$. When $|F_v|>1$, because of the drawing, faces of $F_v$ forms a path in $T$. Suppose that the claim is not true, and let us choose the two distinct faces $f_1, f_2 \in F_v$ which attains the minimum. We cannot have $f_0 \in F_v$, since there is only one bounded face $f$ with $d_T(f,f_0)=0$, which is $f_0$. Furthermore, in the unique path in $T$ from $f_0$ to $f_1$, we cannot use any faces in $F_v \setminus \{f_1\}$, otherwise this contradicts the minimality assumption. The unique path in $T$ from $f_0$ to $f_2$ then can be obtained by concatenating the path from $f_0$ to $f_1$ and the path from $f_1$ to $f_2$ inside $F_v$. This gives a contradiction with the minimality and proves the claim.

We define the edge coloring $c:E\to \ZZ$ by $c(e)= \min_{f\in V(T),\, e\subset f} d_T(f,f_0)$. We show that this is a desired coloring. Fix a vertex $v\in V$, and choose the unique face $f_v \in F_v$ from the claim. We can partition the cycle $C_v$ induced by $F_v \cup \{O\}$ in $G^*$ (note that $O$ is the outerface and $C_v$ might consist of two parallel edges) into $P_1$ and $P_2$ where $f_v$ and $O$ are the common end points of $P_1$ and $P_2$. We also partition $E_v$ into $E_1\cup E_2$ where $E_i$ is the set of edges which have their dual edges in $P_i$ for every $i\in [2]$. Note that, by the uniqueness of a path between two vertices in a tree, the value $c(e)$ only increases from $d_T(f_v, f_0)$ where $e$ is the edge we encounter when we move from $f_v$ via $P_i$ for each $i \in [2]$. This means that for each $i \in [2]$ a color is used for $E_i$ at most once, so any color is used at most twice in $E_v$. Also, note that $c$ gives an improper interval edge coloring. This completes the proof for this case.

\medskip 

\textbf{(Case 2) When $G$ misses some edges of $Q$:} We add all the boundary edges of $Q$ to $G$ to obtain $G'=(V,E')$, and apply the argument of Case 1 for $G'$ to obtain an improper interval edge coloring $c':E'\to \ZZ$. We claim that $c=c'|_{E}$ is the desired coloring.

Again, fix $v\in V$, consider the edges in $E_v$ and compare it with $E_v'$ which is the set of edges of $E'$ incident with $v$. By the argument of Case 1, we obtain a partition $E_v'=E_1'\cup E_2'$ by following the two dual paths where for each $i \in [2]$ all the edges of $E_i'$ are uniquely colored by $c'$. The only difference between $E_v$ and $E_v'$, or $c$ and $c'$, might occur at the boundary of $Q$, or equivalently, at the edge which attains the maximum $\max_{e \in E_i'} c'(e)$ for each $i \in [2]$. Therefore, excluding some of these edges and considering the restricted coloring $c$ still gives the desired result. This completes the proof.
\end{proof}

\section{A construction of $k$-trees with unbounded impropriety} \label{sec_k-tree}
\begin{figure}[h]
	\centering
	\includegraphics[totalheight=5cm]{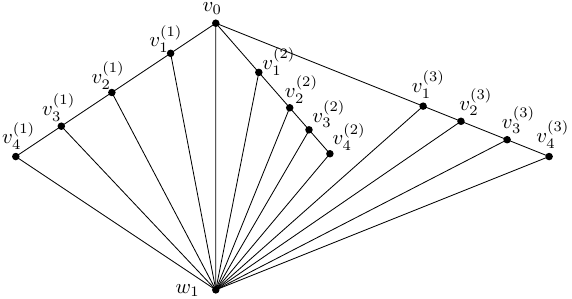}
	\caption{Illustration of $T_{3,4}^{(2)}$.
	}
	\label{fig:2-tree}
\end{figure}
In this section we prove Theorem \ref{thm_2-tree}. We first introduce our construction. For three positive integers $k\geq 2$, $m$ and $n$, let $T_{m,n}^{(k)}$ be the graph defined by
\begin{align*}
V(T_{m,n}^{(k)})=&\{w_1, \dots, w_{k-1},v_0\} \cup \{v_i^{(j)}: i\in [n],\, j \in [m] \},\textrm{ and}\\
E(T_{m,n}^{(k)})=&\{v_0v_1^{(j)}: j \in [m]\} \cup  \{v_i^{(j)}v_{i+1}^{(j)}: i\in [n-1],\, j\in [m]\}\\
&\cup \bigcup_{i \in [k-1]}\{w_ix: x \in V(T_{m,n}^{(k)})\setminus \{w_i\}\}.
\end{align*}
In other words, $T_{m,n}^{(k)}$ is constructed by first subdividing each edge of the $m$-star $K_{1,m}$, with the center vertex $v_0$, into a path consisting of $n$ edges, and then adding an edge between each   of the new vertices $w_1, \dots, w_{k-1}$ which form a clique and each of the other vertices in $V(T_{m,n}^{(k)})$. Clearly, $T_{m,n}^{(k)}$ is a $k$-tree.

Note that the essential ingredient in the argument below is to exploit the higher degree of $w_i$ for $i\in [k-1]$ (or simply just $w_1$) than the degrees of the other vertices.

\smallskip 

Before we prove Theorem \ref{thm_2-tree}, we state the following easy observation as a lemma.

\begin{lemma} \label{lemma_distance}
Let $G=(V,E)$ be a graph and $c:E \to \ZZ$ be an improper interval edge coloring. For a walk $p_0p_1\dots p_lp_{l+1}$ in $G$, we have
\[|c(p_0p_1)-c(p_lp_{l+1})|\leq \sum_{i=1}^l (\deg_G(p_i)-1).\]
\end{lemma}
\begin{proof}
By the triangle inequality, we have
\begin{align}
|c(p_0p_1)-c(p_lp_{l+1})|\leq \sum _{i=1}^l |c(p_{i-1}p_i)-c(p_ip_{i+1})|. \label{ineq_proof_lemma}
\end{align}
Since $c$ is an improper interval edge coloring, for any two (not necessarily distinct) edges $e, e'\in E$ sharing a common endpoint $v \in V$, we have $|c(e)-c(e')|\leq \deg_G(v)-1$. This implies the desired conclusion from (\ref{ineq_proof_lemma}).
\end{proof}

\begin{proof}[Proof of Theorem \ref{thm_2-tree}]
Fix positive integers $k \geq 2$ and $N$. For positive integers $m$ and $n$ which are to be determined later, let us denote $T_{m,n}^{(k)}$ briefly by $G$. Let $c:E(G) \to \ZZ$ be an arbitrary improper interval edge coloring of $G$. It is sufficient to show that for suitable values of $m$ and $n$ there is a color $l \in \ZZ$ such that $|c^{-1}(l) \cap E_{w_1}|\geq N$.

\smallskip 

By translating $c$ if necessary, we can assume that $c(w_1v_0)=0$. We want to measure $|c(v_i^{(j)}w_1)|$ for every $i\in [n]$ and $j\in [m]$. Applying Lemma \ref{lemma_distance} to the walk $w_1v_0v_1^{(j)}\dots v_i^{(j)}w_1$, we have 
 \begin{align}
|c(v_i^{(j)}w_1)|=&|c(w_1v_0)-c(v_i^{(j)}w_1)|\leq (\deg_G(v_0)-1)+\sum_{p=1}^i (\deg_G(v_p^{(j)})-1)\nonumber\\
\leq& (m+k-2)+ik\leq m+(n+1)k-2 \label{bound_final}
\end{align}
for every $i\in [n]$ and $j\in [m]$.
By (\ref{bound_final}), we can conclude that for any $i\in [n]$ and $j \in [m]$, $c(v_i^{(j)}w_1)$ belongs to the integral interval $[-m-kn-(k-2), m+kn+(k-2)]$ consisting of $2m+2kn+(2k-3)$ integers. 

Now, we let $m=n$ and $E_1=\{v_i^{(j)}w_1: i,j\in [n]\}$.
Then $|E_1|=n^2$ and the edges of $E_1$ use at most $(2k+2)n+(2k-3)\leq (2k+3)n$ colors for $n \geq 2k-3$. Therefore, with a sufficiently large $n$ satisfying $\frac{n}{2k+3}\geq N$, there is a color used for at least $N$ edges in $E_1$ by the pigeonhole principle. This completes the proof.
\end{proof}
Certainly, one can use different values of $n$ and $m$ to have the same desired conclusion. We did not try to find the smallest $T_{m,n}^{(k)}$ which satisfies the conclusion.

\section*{Acknowledgement}
The author learned the above conjectures from Eun-Kyung Cho at the 98th KPPY workshop held in Kyungpook National University, South Korea. He is deeply indebted to her for helpful and encouraging discussions since then. He is also grateful to the organizers of the KPPY workshop series.

Last but not least, he wishes to thank Carl Johan Casselgren and Petros A. Petrosyan for their insightful comments and discussions.

\bibliographystyle{plain}
\bibliography{bibliography}

\end{document}